\documentclass[a4paper,12pt]{amsart} 

\usepackage[lmargin=3.7cm,rmargin=2.7cm,bmargin=3.5cm,tmargin=3cm]{geometry}

\usepackage[latin1]{inputenc}
\usepackage{amsmath}
\usepackage{amssymb}  
\usepackage{amsthm}
\usepackage{textcomp}
\usepackage{enumerate}
\usepackage[sort,numbers]{natbib}
\usepackage{array}
\usepackage{bbm}
\usepackage{verbatim}

\usepackage{color} 

\newcommand{\defeq}{\mathrel{\mathop:}=}

\newcommand{\F}{\mathcal{F}}

\newcommand{\uarg}{\,\cdot\,}
\newcommand{\ud}{\mathrm{d}}
\newcommand{\R}{\mathbb{R}}
\newcommand{\N}{\mathbb{N}}
\renewcommand{\P}{\mathbb{P}}
\newcommand{\E}{\mathbb{E}}
\newcommand{\B}{\mathcal{B}}
\newcommand{\charfun}[1]{\mathbbm{I}\left\{#1\right\}}

\newcommand{\given}{\,:\,}

\newtheorem{theorem}{Theorem}
\newtheorem{corollary}[theorem]{Corollary}
\newtheorem{proposition}[theorem]{Proposition}
\newtheorem{lemma}[theorem]{Lemma}

\theoremstyle{definition}
\newtheorem{definition}[theorem]{Definition}

\newtheorem{condition}[theorem]{Condition}

\theoremstyle{remark}
\newtheorem{remark}[theorem]{Remark}

\title{Quantitative convergence rates for sub-geometric Markov chains}

\author{Christophe Andrieu}
\address{Christophe Andrieu, School of Mathematics, University of Bristol, BS8 1TW, United
Kingdom}
\email{C.Andrieu@bristol.ac.uk}

\author{Gersende Fort}
\address{Gersende Fort, 
  CNRS \& Telecom ParisTech,
  46 rue Barrault
  75634 Paris Cedex 13
  France}
\email{gersende.fort@telecom-paristech.fr}

\author{Matti Vihola}
\address{Matti Vihola, Department of Mathematics and Statistics, University of
Jyväskylä, P.O.Box 35, FI-40014 Univ.~of Jyväskylä, Finland}
\email{matti.vihola@iki.fi}

\keywords{Markov chain, 
  inhomogeneous,
  polynomial ergodicity,
  sub-geometric ergodicity}

\subjclass[2010]{Primary 
  60J05; 
  secondary 
  60J22, 
  }

\begin{document}

\maketitle

\begin{abstract} 
We provide explicit expressions for the constants involved in the 
characterisation of ergodicity of sub-geometric Markov chains. The
constants are determined in terms of those appearing in the assumed 
drift and one-step minorisation conditions. The result is fundamental
for the study of some algorithms where uniform bounds for these
constants are needed for a family of Markov kernels. Our result
accommodates also some classes of inhomogeneous chains.
\end{abstract} 

\section{Introduction}
\label{sec:intro} 

Quantitative convergence rates of Markov chains have been extensively
studied in the geometric ergodicity scenario; see, for example,
\cite{baxendale} and \cite{douc-moulines-rosenthal} and references
therein for homogeneous and inhomogeneous Markov chains, respectively.
Such results have proved to be very useful in certain applications,
such as the analysis of adaptive Markov chain Monte Carlo (MCMC) or
stochastic approximation (SA) recursions
\cite[e.g.][]{andrieu-thoms,fort-moulines-priouret,andrieu-moulines-priouret},
where quantifying the convergence rates of a family of Markov kernels
$\{P_\theta\}_{\theta\in\Theta}$ in terms of $\theta\in\Theta$ is
required. In some cases, delicate control of the constants can also
be used to deduce the stability of the algorithms
\cite[e.g.][]{saksman-vihola,andrieu-vihola}. 

In the present work, we establish explicit bounds on the rate of
convergence of sub-geometric Markov chains in terms of the constants
involved in standard drift and minorisation conditions. As in
the geometric context, such results are important for adaptive MCMC
and SA with sub-geometric kernels \cite[e.g.][]{atchade-fort}. 
In section \ref{sec:applications} we discuss in more details two specific applications prompted by two other recent
methodological and theoretical developments in the area of MCMC \cite{maire-douc-olsson,andrieu-vihola-pseudo}.

We now provide a brief discussion of existing results and how they
relate with our work. Hereafter, we shall use the following standard
notation whenever well-defined:
\begin{align*}
    P f(x) &\defeq \textstyle \int P(x,\ud y) f(y),
    & \mu(f)&\defeq \textstyle\int \mu(\ud x) f(x), \\
    P Q(x,A) & \defeq \textstyle \int P(x,\ud y) Q(y,A), 
    &\mu P(A) &\defeq \textstyle \int \mu(\ud x) P(x,A),
\end{align*}
where $P$ and $Q$ are Markov kernels on a measurable space
$(\mathsf{X},\B(\mathsf{X}))$, $f:\mathsf{X}\to\R$ is a measurable
function and $\mu$ is a (signed) measure. 

In the literature, the Markov chain
`convergence rate' often refers to the rate of convergence of marginal
distributions, that is, if $\pi$ is the invariant measure of $P$,
\begin{equation}
    \hat{r}(n) | P^n f(x) - \pi(f) | \le c V(x) \qquad \text{for all $n\in
      \N$ and
    $x\in\mathsf{X}$},
  \label{eq:marginal-rate}
\end{equation}
where $\big(\hat{r}(n)\big)_{n\ge 0}$ is a positive non-decreasing rate
sequence, $f$ belongs to a suitable class of functions
integrable respect to $\pi$, the function 
$V:\mathsf{X}\to[1,\infty)$ is measurable and $c$ is a finite constant
which is often left unspecified. We focus here instead on establishing
the stronger property
\begin{equation} 
    \sum_{n=0}^\infty r(n) |P^n f(x) - \pi(f) | \le c
    V(x)\qquad\text{for all $x\in\mathsf{X}$,}
    \label{eq:sum-rate}
\end{equation}
and aim to quantify the constant $c$ in terms of the constants in
Condition \ref{cond:drift-for-coupling}. 
The rate $\big(r(n)\big)_{n\ge 0}$ is positive non-decreasing as
$\big(\hat{r}(n)\big)_{n\ge 0}$, and if $r(n)=\hat{r}(n)$,
\eqref{eq:sum-rate} clearly implies \eqref{eq:marginal-rate}.
While the distinction between
\eqref{eq:marginal-rate} and \eqref{eq:sum-rate} is often not
essential in the geometric case, it turns out to be important in some
sub-geometric scenarios. Indeed, for some applications, using the
marginal convergence rate \eqref{eq:marginal-rate} to deduce a
property of the type \eqref{eq:sum-rate} may be sub-optimal for
sub-geometric Markov chains; an example is briefly discussed below.

The characterisation of sub-geometric Markov chains with drift and 
minorisation conditions has been considered in various earlier works
starting with the pioneering work of Tuominen and Tweedie
\cite{tuominen-tweedie}. In the more recent works Fort and Moulines
\cite{fort-moulines-polynomial} and Jarner and Roberts
\cite{jarner-roberts} establish polynomial rates of convergence,
but do not provide quantitative results. Douc, Fort, Moulines and
Soulier \cite{douc-fort-moulines-soulier} (see also \cite{fort-phd})
have extended these results to more general sub-geometric ergodicity
scenarios. The latter works consider quantities of the type
\eqref{eq:sum-rate}, but do not provide a quantitative expression for
the constant $c$.  

Douc, Moulines and Soulier \cite{douc-moulines-soulier} have later
provided rates of convergence for sub-geometric chains with computable
constants, but their approach is restricted to the convergence of the
marginals \eqref{eq:marginal-rate} and no result is available
concerning \eqref{eq:sum-rate}. Although bounds of the form
\eqref{eq:marginal-rate} may imply \eqref{eq:sum-rate} in some
scenarios, such an approach may be sub-optimal and lead to a
significant loss. This is the case, for example, with certain
polynomial kernels yielding \eqref{eq:marginal-rate} with rate 
$\hat{r}(n)\propto n^\beta$ with some $\beta>0$ 
\cite{jarner-roberts}. This guarantees the finiteness of
the sum in \eqref{eq:sum-rate} with a constant rate $r(n)=1$
only if $\beta>1$, 
whereas our results imply \eqref{eq:sum-rate} 
also with weaker polynomial rates including the cases $\beta\in(0,1]$ 
of \cite{jarner-roberts}.

Our main result, Theorem \ref{thm:generic-explicit-rates}
in Section \ref{sec:main-result}, provides an explicit 
upper bound for the constant $c$ 
for a slight generalisation of \eqref{eq:sum-rate}. The
approach follows that of Andrieu and Fort
\cite{andrieu-fort-explicit}, but we complement it by providing
explicit and relatively simple expressions, valid under a slightly
stronger but more easily applicable one-step minorisation condition.
In Section \ref{sec:poly} we then establish a set of corollaries of
Theorem \ref{thm:generic-explicit-rates} for the important special
case of polynomially ergodic chains, and continue with discussion on
two specific applications in Section \ref{sec:applications}.  The
proof of Theorem \ref{thm:generic-explicit-rates} is given in Section
\ref{sec:proof}, after describing the notation and definitions in
Section \ref{sec:coupling}.  Our proof is nearly self-contained, using
only two auxiliary results which are restated in Appendix
\ref{sec:literature} for the reader's convenience.


\section{Explicit rate of convergence for sub-geometric Markov chains} 
\label{sec:main-result} 

We start by the generic main assumption, a sub-geometric drift 
condition towards a small set, and recall the definition of Young
functions.

\begin{condition} 
    \label{cond:drift-for-coupling} 
Suppose $(P_k)_{k\ge 1}$ is a collection of Markov kernels on a measurable space 
$(\mathsf{X},\B(\mathsf{X}))$. Assume there exist a set
$C\in\B(\mathsf{X})$, a measurable function $V:\mathsf{X}\to[1,\infty)$ 
and a concave, non-decreasing and differentiable function
$\phi:[1,\infty)\to(0,\infty)$ such that 
$\lim_{t\rightarrow\infty}\phi'(t)=0$. 
Moreover, there exist probability measures
$(\nu_k)_{k\ge 1}$ on $(\mathsf{X},\B(\mathsf{X}))$ 
and constants $\epsilon_\nu,\epsilon_b \in(0,1)$, 
$b_V,c_V<\infty$ such that for all $k\ge 1$,
\begin{align*}
P_k V(x) &\le V(x)-\phi\circ V(x)+b_V\charfun{x\in C}  \\
P_k(x,\uarg) & \ge \epsilon_\nu \nu_k(\uarg) \qquad \text{for all }x\in C \\
\inf_{x\notin C} \phi \circ V(x) &\ge
b_V(1-\epsilon_b)^{-1} \qquad\text{and}\qquad
\sup_{x\in C} V(x) \le c_V.
\end{align*}
\end{condition} 

\begin{definition} 
    \label{def:young} 
The non-decreasing functions $\Psi_{1},\Psi_{2}:[1,\infty)\to(0,\infty)$ are (a pair
of) Young functions if $\Psi_{1}(x)\Psi_{2}(y)\leq x+y$
for all $x,y\ge 1$.
\end{definition} 

Theorem \ref{thm:generic-explicit-rates} 
when applied with $P_k = P$ is a refinement of Proposition 3.1 and 
Theorem 3.6 in \cite{andrieu-fort-explicit} since
it provides an explicit expression of the upper bound.

\begin{theorem}
\label{thm:generic-explicit-rates} 
Assume Condition \ref{cond:drift-for-coupling}.
Then there exists
a constant $c\in[0,\infty)$ dependent on $b_V$, $c_V$, 
$\epsilon_b$, $\epsilon_\nu$ and $\phi$ only,
such that for any pair of Young functions $\Psi_1$, $\Psi_2$
and any measurable $f:\mathsf{X}\to\R$ satisfying
$\|f\|_{W} \defeq \sup_{x\in\mathsf{X}} |f(x)|/W(x)<\infty$ with
$W(x) \defeq
\Psi_{2}\big(\phi\circ V(x)/\phi(1)\big)$,
\begin{equation}
\sum_{n\geq0}\Psi_{1}\big(r(
n)\big)|P^{(n)}f(x)-P^{(n)}f(x')|\leq
c(V(x)+V(x')-1) \|f\|_{W},
\label{eq:coupling-bound}
\end{equation}
where $P^{(n)} \defeq P_1\cdots P_n $, with the convention
$P^{(0)}(x,A) \defeq \charfun{x\in A}$, the indicator function, and
where $r:\N\to[1,\infty)$ is defined through
$H_\phi:[1,\infty)\to[0,\infty)$ by 
\begin{equation}
    H_\phi(t)\defeq \int_{1}^{t}\frac{\ud s}{\phi(s)}
    \qquad\text{and}\qquad
    r(n) \defeq \frac{\phi \circ H_\phi^{-1}(\epsilon_b n)}{\phi(1)}.
    \label{eq:def-r-n}
\end{equation}

The constant $c$ can be given as
\begin{align*}
    c &\defeq \frac{2}{\epsilon_b\phi(1)}\bigg[ 
    2 + \frac{\bar{b}}{\epsilon_\nu}
    + c_* \bar{b}
      r(1)\bigg(1 + \frac{r(1)}{\epsilon_b \phi(1)}
    \bigg)\bigg],
\end{align*}
where
\begin{align*}
    \bar{b} &\defeq 2b_V + \epsilon_b\phi(1),
    & c_* &\defeq \sum_{j=1}^\infty
    (1-\epsilon_\nu)^{j-1}\prod_{k=1}^{j-1}
    \big(1+\delta_kM_1\big), \\
    \delta_k &\defeq \epsilon_b(\phi'\circ H_\phi^{-1})(\epsilon_b k),
    & M_1 &\defeq r(1)\bigg[1 + \frac{2r(1)}{\epsilon_b \phi(1)} 
    \bigg(\frac{b_V+c_V}{1-\epsilon_\nu} -1\bigg)\bigg].
\end{align*}
\end{theorem} 

The proof of Theorem \ref{thm:generic-explicit-rates} is postponed to
Section \ref{sec:proof}. 

\begin{remark} 
In Theorem \ref{thm:generic-explicit-rates},
\begin{enumerate}[(i)]
\item it is easy to see that the assumptions imply 
$\lim_{t\to\infty}H_\phi^{-1}(t)=\infty$ so
$\lim_{k\to\infty} \delta_k = 0$ and therefore $c_*<\infty$.
\item in the case of a constant drift, that is, 
  if the function
  $\phi \equiv \epsilon_\phi>0$, then we have  $c_* = \epsilon_\nu^{-1}$.
\item the condition $\inf_{x\notin C} \phi\circ V(x) \ge b_V/(1-\epsilon_b)$
  is essential for our proof; we need the bivariate drift established in
  Lemma \ref{lem:joint-drifts}. If $\lim_{t\to\infty}\phi(t)=\infty$, 
  it is often possible to check Condition
  \ref{cond:drift-for-coupling}; see Corollary 
  \ref{cor:practical-interpolated-drifts} for the polynomial case.
\item \label{item:generic-initial-measures} 
  if $\mu_1$ and $\mu_2$ be probability 
  measures such that $\mu_1(V)+\mu_2(V)<\infty$, then
  \eqref{eq:coupling-bound} implies the following bound,
  \begin{equation*}
    \sum_{n\geq0}\Psi_{1}\big(r(n)\big)
    |\mu_1 (P^{(n)}f)-\mu_2 (P^{(n)} f)|\leq
    c(\mu_1(V)+\mu_2(V)-1) \|f\|_W,
  \end{equation*}
  because for any function $g$ integrable with respect to 
  $\mu_1$ and $\mu_2$, we have 
  $| \mu_1(g)-\mu_2(g)|\le \int \mu_1(\ud x) \mu_2(\ud x')
  |g(x)-g(x')|$.
\item 
  suppose that $\pi$ is the invariant probability measure of
  $P_k$ for $k\ge 1$ and $\pi(V)<\infty$, then 
  \eqref{item:generic-initial-measures} with
  $\mu_1 = \charfun{x\in \cdot}$ and $\mu_2=\pi$
  yields
\[
\sum_{n\geq0}\Psi_{1}\big(r(n)\big)|P^{(n)}f(x)-\pi(f)|\leq
c(V(x)+\pi(V)-1)\|f\|_W.
\]
\item it is possible to refine the bound by
  replacing the term $c(V(x)+V(x')-1)$ with $c_1(V(x)+V(x'))+c_2$,
  where the constants $c_1$ and $c_2$ are easily accessible
  from the statements of Lemmas \ref{lem:phi-v} and 
  \ref{lem:r-phi-sum}.
\end{enumerate}
\end{remark} 


\section{Rate of convergence for polynomially ergodic chains} 
\label{sec:poly} 

We state here two convenient corollaries of Theorem 
\ref{thm:generic-explicit-rates} in the case where
$P$ satisfies a polynomial drift condition.
The first corollary characterises the required balance between the class of functions
and the rate of convergence.

\begin{corollary} 
    \label{cor:poly} 
Assume Condition \ref{cond:drift-for-coupling} holds with
$\phi(v) = \beta v^{\alpha}$ with some constants 
$\beta>0$, $\alpha\in[0,1)$ and $\epsilon_b\in(0,1)$. 
Then, for any $\xi\in[0,1]$ and for any measurable
function $f:\mathsf{X}\to\R$ with 
\[
    \|f\|_{V^{\alpha(1-\xi)}} \defeq \sup_{x\in\mathsf{X}}
\frac{|f(x)|}{V^{\alpha(1-\xi)}(x)}<\infty,
\]
there exists a constant $c_{\alpha,\beta,\epsilon_b,\xi}<\infty$ depending on
$\alpha$, $\beta$, $\epsilon_b$ and $\xi$ such that
\begin{equation}
    \sum_{n\ge 0} (n+1)^{\frac{\xi\alpha}{1-\alpha}}
    | P^{(n)} f(x) - P^{(n)} f(x') | \le c_{\alpha,\beta,\epsilon_b,\xi} c
    \|f\|_{V^{\alpha(1-\xi)}}(V(x)+V(x')-1),
    \label{eq:polynomial-bound}
\end{equation}
where $c<\infty$ is the constant 
given in Theorem \ref{thm:generic-explicit-rates}.
\end{corollary} 
\begin{proof} 
We may compute
\[
    H_\phi(t) = \int_1^t \frac{\ud s}{\beta t^{\alpha}}
    = \frac{t^{1-\alpha} - 1}{\beta(1-\alpha)} 
    \quad\text{and}\quad
    H_\phi^{-1}(n) = (n\beta(1-\alpha)+1)^{\frac{1}{1-\alpha}},
\]
so we obtain
\[
    r(n) =
    \big(\epsilon_b n\beta(1-\alpha)+1)^{\frac{\alpha}{1-\alpha}}
    \ge (n+1)^{\frac{\alpha}{1-\alpha}}c_{\alpha,\beta,\epsilon_b},
\]
where $c_{\alpha,\beta,\epsilon_b} \defeq
    \min\{1,(\epsilon_b\beta(1-\alpha))^{\frac{\alpha}{1-\alpha}}\}.$
Define the functions
\begin{equation}
    \Psi_1(x) \defeq \begin{cases}\xi^{-1}x^{\xi},&\xi\in(0,1)\\
    x,&\xi=1 \\
    1,&\xi=0
    \end{cases}
    \qquad
    \Psi_2(y) \defeq \begin{cases}(1-\xi)^{-1} y^{1-\xi},&\xi\in(0,1)\\
    1,&\xi=1 \\
    y,&\xi=0,
    \end{cases}
    \label{eq:def-psi-polynomial}
\end{equation}
satisfying $\Psi_1(x)\Psi_2(y)\le x+y$, by Young's inequality for
$\xi\in(0,1)$. 

Theorem \ref{thm:generic-explicit-rates} implies that 
\[
    \sum_{n\ge 0}
    \Psi_1\big(c_{\alpha,\beta,\epsilon_b}(n+1)^{\frac{\alpha}{1-\alpha}}\big)
    | P^{(n)} f(x) - P^{(n)} f(x') | \le c
    \|f\|_{\Psi_2(\beta V^{\alpha})}(V(x)+V(x')-1),
\]
from which we deduce the claim with
\[
    c_{\alpha,\beta,\epsilon_b,\xi} = c_{\alpha,\beta,\epsilon_b}^{-\xi} 
    \big[
    (1-\xi)\xi\beta^{\xi-1}
    + \charfun{\xi=1} + \charfun{\xi=0}\beta^{-1}\big].
    \qedhere
\]
\end{proof} 

We further consider a corollary which allows one to consider different 
growth rates of the upper bound in \eqref{eq:polynomial-bound} in terms
of $x$ and $x'$.

\begin{condition}
    \label{cond:uniform-polynomial-drift-and-minorisation} 
Suppose $\mathcal{P}$ is a collection of Markov kernels on  
$(\mathsf{X},\B(\mathsf{X}))$. Assume there exist a set
$C\in\B(\mathsf{X})$ and a measurable function $\hat{V}:\mathsf{X}\to[1,\infty)$ 
with $c_{\hat{V}}\defeq \sup_C \hat{V}<\infty$ and constants $\beta>0$, $\alpha\in (0,1)$ and
$b_{\hat{V}}<\infty$ such that for all $P\in\mathcal{P}$
\begin{align*}
    P\hat{V}(x) &\le \begin{cases} 
    \hat{V}(x) - \beta \hat{V}^\alpha(x), & x\notin C, \\ 
    b_{\hat{V}}, & x\in C.
    \end{cases}
\end{align*}
Furthermore, suppose that every level set $A_{\hat{V}}(v)\defeq
\{x\in\mathsf{X}: \hat{V}(x)\le v\}$ is uniformly 1-small, that is,
there exist $\epsilon_v>0$ and probability measures
$(\nu_P)_{P\in\mathcal{P}}$ on
$\big(\mathsf{X},\B(\mathsf{X})\big)$ such that for all $P\in\mathcal{P}$
\[
    P(x,\uarg) \ge \epsilon_v \nu_P(\uarg) \qquad\text{for all
      $x\in A_{\hat{V}}(v)$.}
\]
\end{condition} 

We first observe that Condition 
\ref{cond:uniform-polynomial-drift-and-minorisation} implies
Condition \ref{cond:drift-for-coupling} for functions $V=\hat{V}^\eta$
with any $\eta\in(1-\alpha,1]$.

\begin{proposition} 
    \label{prop:drift-with-moments} 
Suppose Condition \ref{cond:uniform-polynomial-drift-and-minorisation}
holds. Then, for any $(P_k)_{k\ge 1}\subset \mathcal{P}$ and $\lambda\in [0,1)$, 
Condition \ref{cond:drift-for-coupling} holds with $V(x) =
\hat{V}(x)^{1-\lambda\alpha}$,
$\phi(v) = (1-\lambda\alpha) \beta v^{\alpha_\lambda}$
where $\alpha_\lambda \defeq \frac{\alpha(1-\lambda)}{1-\lambda \alpha}$,
with the set $C \defeq A_{\hat{V}}(c_V)$, and
with some constants $\epsilon_b,\epsilon_\nu\in(0,1)$ and
$b_V, c_V<\infty$, whose values depend only on $\lambda$ and
the constants and the function $\hat{V}$ in
Condition \ref{cond:uniform-polynomial-drift-and-minorisation}.
\end{proposition} 
\begin{proof} 
Let $P\in\mathcal{P}$. 
Following the proof of \cite[Lemma 3.5]{jarner-roberts}, 
Jensen's inequality and the mean value theorem imply
with $\eta = 1-\lambda \alpha$
\begin{align*}
    P\hat{V}^\eta(x) &\le (\hat{V}-\beta \hat{V}^\alpha(x))^\eta
    \le \hat{V}^\eta(x) - \eta \beta \hat{V}^{\eta\alpha_\lambda}(x)
    &&x\notin C \\
    P\hat{V}^\eta(x) &\le b_{\hat{V}}^\eta && x\in C,
\end{align*}
where $\alpha_\lambda = \frac{\alpha-(1-\eta)}{\eta}
= \frac{\alpha(1-\lambda)}{1-\lambda \alpha}$.
Clearly
\begin{align}
    PV(x) &\le V(x) - \phi\circ V(x)
    + b_V  \charfun{x\in C},
    \label{eq:required-drift}
\end{align}
where $\phi(v) = \eta \beta v^{\alpha_\lambda}$ and $b_V =
b_{\hat{V}}^\eta + \phi(c_{\hat{V}})$. Let $\epsilon_b\in(0,1)$ and
take $c_V \in[c_{\hat{V}},\infty)$ sufficiently large so that
$\phi(c_V)\ge b_V(1-\epsilon_b)^{-1}$.
\end{proof} 

\begin{corollary} 
    \label{cor:practical-interpolated-drifts} 
Suppose Condition \ref{cond:uniform-polynomial-drift-and-minorisation}
holds.
Then, for any $\xi\in[0,1]$ and $\lambda\in[0,1)$, there exists
a constant $c_{*}<\infty$ such that for
all $(P_k)_{k\ge 1}\subset \mathcal{P}$ and
$\|f\|_{\hat{V}^{\alpha_{\lambda,\xi}}}<\infty$
where $\alpha_{\lambda,\xi} = \alpha(1-\lambda)(1-\xi)$,
\[
    \sum_{n\ge 0} (n+1)^{\frac{\alpha(1-\lambda)\xi}{1-\alpha}}
    | P^{(n)} f(x) - P^{(n)} f(x') | \le c_{*} 
    \|f\|_{\hat{V}^{\alpha_{\lambda,\xi}}}(\hat{V}^{1-\lambda \alpha}(x)
             +\hat{V}^{1-\lambda \alpha}(x')-1).
\]
\end{corollary} 
\begin{proof} 
Proposition \ref{prop:drift-with-moments}
and Corollary
\ref{cor:poly} imply the claim.
\end{proof} 


\section{Applications} 
\label{sec:applications} 

We discuss next two specific applications of our results. Both 
applications are related to the evaluation of the efficiency of Markov chain
Monte Carlo (MCMC) schemes in terms of asymptotic
variance: the first application involves so-called pseudo-marginal MCMC 
\cite{andrieu-roberts,andrieu-vihola-pseudo}, while the second application
is related to a general comparison result 
of inhomogeneous Markov chains recently established in \cite{maire-douc-olsson}.

In both cases, one is interested in estimating an integral 
\[
    \pi(f) \defeq \int_{\R^d} f(x) \pi(x) \ud x, 
\]
where $\pi(x)$ is a probability density
and $f$ is a $\pi$-integrable function. The efficiency criterion 
is the so-called asymptotic variance
\[
    \sigma^2(f) \defeq \lim_{n\to\infty} \E\bigg[ \frac{1}{\sqrt{n}}
    \sum_{k=1}^n \big[f(X_k) - \pi(f)\big]\bigg]^2,
\]
where $(X_k)_{k\ge 0}$ denotes the Markov chain with 
initial distribution $\pi$ and with the same $\pi$-invariant 
transition kernel(s) as the MCMC sampler.

\subsection{Efficiency of pseudo-marginal MCMC}
\label{sec:pseudo-application} 

The pseudo-margi\-nal algorithm is relevant to situations where
the density $\pi$ cannot be evaluated point-wise,
which prevents a straightforward implementation of 
the Metropolis-Hastings algorithm for example.
Such a situation occurs naturally, for instance 
when $\pi(x)$ is a marginal density
of a higher-dimensional density. 
As pointed out in \cite{andrieu-roberts,andrieu-vihola-pseudo} it is however 
possible to implement a valid (auxiliary variable)
Metropolis-Hastings algorithm in this scenario, by using 
non-negative unbiased estimators of the density values $\pi(x)$.
Interestingly, regardless
of the accuracy of the related estimator, the corresponding Markov
chain will be ergodic with minimal assumptions, and therefore yield
ergodic averages convergent to the integral of interest 
\cite{andrieu-roberts,andrieu-vihola-pseudo}.

However, the efficiency of the algorithm usually depends heavily on
the properties of the estimators of $\pi(x)$. If the accuracy is increased, the
pseudo-marginal algorithm tends to behave in a way similar to the ideal algorithm
for which exact values of $\pi(x)$ are used instead of estimators.
In particular, let $N\ge 1$ be a parameter controlling the accuracy of the estimator
(such as the number of estimators used when using an averaging property to
reduce variability), and let $\sigma_N^2(f)$ be the asymptotic
variance of the related pseudo-marginal algorithm. Then, under general
conditions, $\sigma_N^2(f) \to \sigma^2(f)$ as $N\to\infty$, where
$\sigma^2(f)$ is the asymptotic variance of the ideal algorithm
\cite[Theorem 21]{andrieu-vihola-pseudo}.

The key assumption required for the aforementioned result to hold
is that the integrated
autocorrelation series converge uniformly, that is, 
\begin{equation}
        \lim_{n\to\infty} \sup_{N\ge 1}\bigg|
        \sum_{k=n}^\infty 
        \E[ \bar{f}(\tilde{X}_0^{(N)})\bar{f}(\tilde{X}_k^{(N)})] \bigg|=0,
        \qquad\text{where }\bar{f}(x) = f(x) - \pi(f),
        \label{eq:act-tail}
\end{equation}
and where $(\tilde{X}_k^{(N)})_{k\ge 0}$ corresponds to the
Markov chain generated by the pseudo-marginal chain with accuracy parameter $N$.

The condition in \eqref{eq:act-tail} is relatively straightforward to check
whenever the pseudo-marginal algorithms are geometrically ergodic 
with uniformly bounded drift and minorisation constants
\cite{baxendale,meyn-tweedie-computable}.
However pseudo-marginal algorithms are sub-geometric
whenever the density estimators of $\pi(x)$ can take arbitrarily large values 
\cite[Proposition 13]{andrieu-vihola-pseudo}.

This is the situation where Corollary \ref{cor:practical-interpolated-drifts} becomes relevant, as it is
straightforward to check \eqref{eq:act-tail} under simultaneous (in $N$) polynomial drift
and minorisation conditions. 
In particular we may write for any $N$ for which the drift and minorisation conditions hold
\[
    \bigg|
        \sum_{k=n}^\infty 
        \E[ \bar{f}(\tilde{X}_0^{(N)})\bar{f}(\tilde{X}_k^N)] \bigg|
    \le \E\bigg[ \big|\bar{f}(X_0^{(N)}\big|
    \sum_{k=n}^\infty \Big| \E\big[\bar{f}(\tilde{X}_k^{(N)})\;\big|\;
    \tilde{X}_0^{(N)}\big]\Big|\bigg]  .
\]
The latter sum can be bounded by Corollary
\ref{cor:practical-interpolated-drifts}; see \cite[Proposition
19]{andrieu-vihola-pseudo} for details. We point out the importance of
having explicit quantitative bounds here in order to ensure that an
upper bound independent of $N$ exists, that is, the constant $c_*$ in
Corollary \ref{cor:practical-interpolated-drifts} can be taken
independent of $N$.


\subsection{Ordering inhomogeneous Markov chains}
\label{sec:inhomogeneous-application} 

In a number of scenarios MCMC algorithms may rely on the composition
of several $\pi$-reversible MCMC kernels. For example when two sampling strategies are available, that is two $\pi$-reversible 
Markov kernels $P_0$ and $Q_0$ can be implemented,  one may consider implementing the algorithm which cycles between these two kernels. The recent result of Maire, Douc and Olsson
\cite[Theorem 4]{maire-douc-olsson} shows that if $P_1$ and $Q_1$ form 
another pair of $\pi$-reversible kernels, and if 
$P_0\preccurlyeq P_1$ and
$Q_0\preccurlyeq Q_1$ in the covariance order, 
then the asymptotic variances related to the two algorithms satisfy
$\sigma_1^2(f) \le
\sigma_0^2(f)$.

The key assumption required by \cite[Theorem 4]{maire-douc-olsson} 
is that the integrated autocorrelation series converges absolutely;
using notation analogous to \eqref{eq:act-tail} 
\begin{equation}
    \sum_{k=1}^\infty \Big(
    \big|\E\big[\bar{f}(X_0^{(i)})\bar{f}(X_k^{(i)})\big]\big|
    + \big|\E\big[\bar{f}(X_1^{(i)})\bar{f}(X_{k+1}^{(i)})\big]\big|
    \Big)<\infty,\qquad i\in\{0,1\},
    \label{eq:maire-douc-olsson}
\end{equation}
where $(X_k^{(i)})_{k\ge 0}$ is the inhomogeneous 
Markov chain with initial distribution $\pi$ and with alternating kernels 
$P_i$ and $Q_i$.

Under geometric ergodicity, \eqref{eq:maire-douc-olsson} 
is relatively easy to check \cite{maire-douc-olsson}.
In the sub-geometric case, we are unaware of any results in the
literature which would be
directly applicable to verify \eqref{eq:maire-douc-olsson}.
When its assumptions are satisfied one can use Theorem
\ref{thm:generic-explicit-rates} to deduce
\eqref{eq:maire-douc-olsson}, exploiting the fact that our results hold
for inhomogeneous Markov chains. In particular, in the polynomial
scenario, Corollary \ref{cor:practical-interpolated-drifts} may be
applied following the arguments in \cite[Proposition
19]{andrieu-vihola-pseudo}. 



\section{Definitions: Coupling and bivariate drift} 
\label{sec:coupling} 

\begin{definition}[Coupling construction] 
    \label{def:coupling-kernel} 
Assume Condition \ref{cond:drift-for-coupling},
denote $\bar{C}=C\times C$ and
define the Markov kernels $\check{P}_k$ on the product space
$(\mathsf{X}\times\mathsf{X},\B(\mathsf{X})\times\B(\mathsf{X}))$
by
\begin{align*}
    \check{P}_k(x,x';A,A')
    \defeq P_k(x,A) P_k(x',A') \charfun{(x,x')\notin \bar{C}}
    + Q_k(x,A) Q_k(x',A') \charfun{(x,x')\in \bar{C}} 
\end{align*}
where $Q_k(x,A) \defeq (1-\epsilon_\nu)^{-1} \big( P_k(x,A) - \epsilon_\nu
\nu_k(A)\big)$.

Define then the Markov
kernels $\bar{P}_k$ on $(\mathsf{X}^2\times\{0,1\},
\B(\mathsf{X})^2\times\mathcal{P}(\{0,1\}))$
as follows for $\check{A}\in\B(\mathsf{X})\times\B(\mathsf{X})$,
\begin{align*} 
   \bar{P}_k(x,x',0; \check{A}\times\{0\})
   &= \big(1-\epsilon_\nu \charfun{(x,x')\in \bar{C}}\big)
   \check{P}_k(x,x'; \check{A}) \\ 
   \bar{P}_k(x,x',0; \check{A}\times\{1\})
   &= \epsilon_\nu \charfun{(x,x')\in\bar{C}} 
   \nu_k(\{x\in\mathsf{X}\given (x,x)\in \check{A}\}) \\
   \bar{P}_k(x,x',1; \check{A}\times\{0\})
   &= 0 \\
   \bar{P}_k(x,x',1; \check{A}\times\{1\})
   &= \charfun{x=x'}\int P_k(x,\ud y) \charfun{(y,y)\in \check{A}}
   + \delta_{(x,x')}(\check{A}).
\end{align*}
Suppose $(X_n,X_n',D_n)_{n\ge 0}$ is a Markov chain defined by the 
kernels $\bar{P}_1,\bar{P}_2,\ldots,\bar{P}_n$
and with $(X_0,X_0',D_0)\equiv (x,x',d)$. We denote 
the probability and the expectation associated with the chain as 
$\P_{x,x',d}$ and 
$\E_{x,x',d}$, respectively, and define the stopping times 
$T_1\defeq \inf\{n\ge 0: (X_{n},X_{n}')\in\bar{C}\}$ 
and $T_k\defeq \inf\{ n> T_{k-1}: (X_{n},X_{n}')\in\bar{C}\}$ for $k\ge 2$,
and $\tau \defeq\inf\{n\ge 0: D_n=1\}$,
with the convention $\inf\emptyset = \infty$.
\end{definition} 

Suppose $D_0\equiv d=0$, then Definition \ref{def:coupling-kernel}
formalises a coupling with probability $\epsilon_\nu$ each time
$(X_n,X_n')\in\bar{C}$; the stopping time $\tau$ is a coupling time,
and $X_{\tau+k}\charfun{\tau<\infty}=X'_{\tau+k}\charfun{\tau<\infty}$
for all $k\ge 0$ $\P_{x,x',0}$-almost surely. If the coupling was not
successful, the chains follow independently $\check{P}_k$ at time $k$ 
until hitting $\bar{C}$ again.

\begin{proposition} 
Consider the Markov chain $(X_n,X'_n,D_n)$ in Definition \ref{def:coupling-kernel}.
Then, $(X_n)_{n\ge 0}$ and $(X'_n)_{n\ge 0}$ follow marginally $P^{(n)}$ and
specifically
\[
    {\P}_{x,x',0}(X_n \in A) = P^{(n)}(x,A)
    \qquad\text{and}\qquad
    {\P}_{x,x',0}(X'_n \in A) = P^{(n)}(x',A),
\]
for all $n\ge 0$, all $(x,x')\in \mathsf{X}^2$
and any $A\in\B(\mathsf{X})$.
\end{proposition} 
\begin{proof} 
It is easy to see that for any $(x,x')\in
\mathsf{X}^2$ and $A\in\B(\mathsf{X})$,
\[
    \bar{P}_k(x,x',0; A\times\mathsf{X}\times\{0,1\})
    = P_k(x,A)
    \quad\text{and}\quad
    \bar{P}_k(x,x',0; \mathsf{X}\times A\times\{0,1\})
    = P_k(x',A),
\]
and $\P_{x,x',0}(X_n=X_n'\mid D_n=1)=1$.
\end{proof} 

\begin{lemma} 
    \label{lem:joint-drifts} 
Assume Condition \ref{cond:drift-for-coupling} and
denote $\bar{V}(x,x') \defeq V(x)+V(x')-1$, then
\begin{align}
    \bar{P}_k \bar{V}(x,x',0) &\le \bar{V}(x,x') - \epsilon_b \phi
    \circ \bar{V}(x,x')
    && (x,x')\notin \bar{C}
    \tag{i}
    \label{eq:joint-drift-coupling1} \\
    \bar{P}_k  \bar{V}(x,x',0) &\le 
    2 (b_V + c_V) - 1
    &&(x,x') \in \bar{C}
    \tag{ii}
    \label{eq:joint-drift-coupling2} \\
    \bar{P}_k \bar{V}(x,x',0) &\le
    \bar{V}(x,x') - \epsilon_b(\phi \circ \bar{V})(x,x') +
    \bar{b}\charfun{(x,x')\in\bar{C}}&&
    (x,x')\in\mathsf{X}^2
    \tag{iii} \label{eq:simple-drift} \\
    \check{P}_k \bar{V}(x,x') &\le 
    2(1-\epsilon_\nu)^{-1} \big(b_V+c_V\big) -1 
    &&(x,x') \in \bar{C},
    \tag{iv}
    \label{eq:joint-drift2}
\end{align}
where $\bar{b} = 2b_V + \epsilon_b\phi(1)$.
\end{lemma} 
\begin{proof} 
Condition \ref{cond:drift-for-coupling}
implies for $(x,x')\notin\bar{C}$, 
\begin{align*}
    \bar{P}_k \bar{V}(x,x',0) 
    &\le \bar{V}(x,x') - \phi\circ V(x) - \phi\circ V(x') 
    + b_V \big(\charfun{x\in C}+\charfun{x'\in C}\big) \\
    &\le \bar{V}(x,x') - \epsilon_b \big(\phi\circ V(x) + \phi\circ
    V(x')\big)
    - (1-\epsilon_b) \inf_{z\notin C}\phi\circ V(z) 
    + b_V \\
    &\le \bar{V}(x,x') - \epsilon_b \phi\circ \bar{V}(x,x'), 
\end{align*}
where the last inequality follows because $\phi$ is convex and
non-decreasing and thus
\begin{equation}
    \phi \circ \bar{V}(x,x') - \phi\circ V(x)
\le \phi\circ V(x') - \phi(1).
\label{eq:convex-est}
\end{equation}
This establishes
\eqref{eq:joint-drift-coupling1}. The bound
\eqref{eq:joint-drift2} follows from
\begin{equation}
    \check{P}_k \bar{V}(x,x') 
    = Q_k V(x) + Q_k V(x') -1
    \le 2 (1-\epsilon_\nu)^{-1} \big(c_V + b_V - \epsilon_\nu
    \nu_k(V)\big) - 1.
    \label{eq:sharp-bound-p-check}
\end{equation}
For 
\eqref{eq:joint-drift-coupling2}, let us write
for $(x,x')\in\bar{C}$,
\[
    \bar{P}_k  \bar{V}(x,x',0) 
    = P_k V(x) + P_k V(x') - 1 
    \le \bar{V}(x,x') - \big(\phi\circ V(x) + \phi\circ V(x')\big) +
    2b_V.
\]
Finally, we turn to \eqref{eq:simple-drift} and observe that the above
inequality with \eqref{eq:convex-est} and
\eqref{eq:joint-drift-coupling1} imply
\begin{align*}
    \bar{P}_k  \bar{V}(x,x',0) 
    &\le \bar{V}(x,x') - \epsilon_b \phi\circ \bar{V}(x,x')
    \charfun{(x,x')\notin \bar{C}} \\
    &\phantom{\le}- \big( \phi\circ\bar{V}(x,x') + \phi(1) +
    2b_V\big)\charfun{(x,x')\in \bar{C}} \\
    &\le \bar{V}(x,x') - \epsilon_b \phi\circ \bar{V}(x,x') \\
    &\phantom{\le}+ \sup_{(x,x')\in\bar{C}} \big[ 2b_V + \phi(1) - (1-\epsilon_b)
    \phi\circ \bar{V}(x,x')\big]\charfun{(x,x')\in \bar{C}}.
\end{align*}
The claim follows noticing that $\phi\circ \bar{V}(x,x')\ge \phi(1)$.
\end{proof} 


\section{Proof of Theorem \ref{thm:generic-explicit-rates}} 
\label{sec:proof} 

We give the skeleton of the proof of Theorem
\ref{thm:generic-explicit-rates} next, 
and postpone bounding the involved terms to lemmas.

\begin{proof}[Proof of Theorem \ref{thm:generic-explicit-rates}] 
It is sufficient to prove the claim assuming
$\|f\|_{W}=1$.
Consider the coupling construction in Definition 
\ref{def:coupling-kernel}.
We may write
\begin{align*}
&\sum_{n\ge 0} \Psi_{1}\big(r(n)\big)|P^{(n)}f(x)-P^{(n)}f(x')| \\
&= \sum_{n\ge 0}
\Psi_{1}\big(r(n)\big)
\big|{\E}_{x,x',0}\big[\big(f(X_n)-f(X'_n)\big)\charfun{\tau>n}\big]\big| \\                                       \
&\le {\E}_{x,x',0}\bigg[\sum_{n= 0}^{\tau-1}
\Psi_{1}\big(r(n)\big) W(X_n) \bigg]
+ {\E}_{x,x',0}\bigg[\sum_{n= 0}^{\tau-1}
\Psi_{1}\big(r(n)\big) W(X_n)\bigg].
\end{align*}
Because $\Psi_1(x)\Psi_2(y)\le x+y$, we obtain the bound
\[
    \sum_{n\ge 0} \Psi_{1}(r(n))|P^{(n)}f(x)-P^{(n)}f(x')|
    \le 2\bigg[E_{\text{1}}(x,x')+\frac{E_{2}(x,x')}{\phi(1)}\bigg],
\]
where the terms on the right are defined as
\begin{align*}
E_{1}(x,x')\defeq{\E}_{x,x',0}\left[\sum_{n=0}^{\tau-1}r(n)\right]
\quad\text{and}\quad
E_{2}(x,x')\defeq
{\E}_{x,x',0}\bigg[\sum_{n=0}^{\tau-1} 
  \phi \circ \bar{V}(X_n,X'_n)\bigg],
\end{align*}
and these terms are bounded by Lemma \ref{lem:phi-v} and
\ref{lem:r-phi-sum} below. 
\end{proof} 

\begin{lemma} 
    \label{lem:log-concavity} 
Let $\phi:[1,\infty)\to(0,\infty)$ be concave, non-decreasing and
differentiable, and let $r(n)$ be as defined in \eqref{eq:def-r-n}.
Then, $r(n)$ is non-decreasing and 
for all $n,m\ge 0$,
\begin{align}
    r(n+m) &\le
r(n)r(m)
\tag{i} \label{eq:log-concavity} \\
    r(n+m) - r(n) & \le \epsilon_b
    (\phi'\circ H_\phi^{-1})(\epsilon_b n) r(n)
    \sum_{k=1}^{m} r(k).
\tag{ii} \label{eq:diff-est} 
\end{align}
\end{lemma} 
\begin{proof} 
Denote $r(t) \defeq (\phi \circ H_\phi^{-1})(\epsilon_b t)/\phi(1)$ for $t\in\R_+$,
and compute
\[
    r'(t) = \epsilon_b (\phi' \circ H_\phi^{-1})(\epsilon_b t)
    r(t) \ge 0
    \qquad\text{and}\qquad
    (\log r)'(t) = \epsilon_b (\phi'\circ H_\phi^{-1})(\epsilon_b t).
\]
The latter is non-increasing, therefore \eqref{eq:log-concavity}
follows from
\begin{align*}
    \log r(n+m) - \log r(n) 
    = \int_n^{n+m} (\log r)'(t) \ud t 
    \le \int_0^{m} (\log r)'(t) \ud t
= \log r(m),
\end{align*}
because $r(0)=1$. By the mean value theorem
\begin{align*}
    r(n+m) - r(n)
    &= \sum_{k=n}^{n+m-1} \big(r(k+1)-r(k)\big) 
    = \sum_{k=n}^{n+m-1} r'(k+\xi_k),
\end{align*}
for some $\xi_n,\ldots,\xi_{n+m-1}\in[0,1]$.
Observe that $r'(k+\xi_k)\le \epsilon_b (\phi'\circ
H_\phi^{-1})(\epsilon_b n)
r(k+1)$, so
\begin{align*}
    r(n+m) - r(n)
    &\le \epsilon_b (\phi'\circ H_\phi^{-1})(\epsilon_b n) 
    \sum_{k=n}^{n+m-1} r(k+1).
\end{align*}
We deduce \eqref{eq:diff-est} by applying \eqref{eq:log-concavity}.
\end{proof} 

\begin{lemma}
    \label{lem:phi-v} 
Assume Condition \ref{cond:drift-for-coupling} and 
consider the coupling construction in
Definition \ref{def:coupling-kernel}. Then,
\begin{align}
\E_{x,x',0}\bigg[\sum_{n=0}^{\tau-1} \phi \circ \bar{V}(X_n,X_n') \bigg]
&\le \frac{1}{\epsilon_b}\bar{V}(x,x') + 
\frac{\bar{b}}{\epsilon_b\epsilon_\nu},
\end{align}                                                                  
where $\bar{b}$ is defined in Theorem \ref{thm:generic-explicit-rates}.
\end{lemma} 
\begin{proof} 
By Lemma \ref{lem:joint-drifts} \eqref{eq:simple-drift}
and Proposition \ref{prop:comparison},
\[
    \E_{x,x',0}\bigg[\sum_{n=0}^{\tau-1} \epsilon_b
    (\phi\circ \bar{V})(X_n,X_n')\bigg]
    \le \bar{V}(x,x') + 
    \bar{b} \E_{x,x',0}\bigg[\sum_{n=0}^{\tau-1}
    \charfun{(X_n,X'_n)\in\bar{C}}\bigg],
\]
and we can write
\[
    \E_{x,x',0}\bigg[\sum_{n=0}^{\tau-1}
    \charfun{(X_n,X'_n)\in\bar{C}}\bigg]
    = \sum_{j=1}^\infty \P_{x,x',0}(\tau> T_j)
    = \sum_{j=1}^\infty (1-\epsilon_\nu)^{j-1},
\]
because $\P_{x,x',0}(\tau>
T_j)=\P_{x,x',0}(D_{T_1+1}=0,\ldots,D_{T_{j-1}+1}=0)$
and $\P_{x,x',0}(D_{T_i+1}=0\mid 
D_{T_{1}+1}=0,\ldots,D_{T_{i-1}+1}=0)=1-\epsilon_\nu$.
\end{proof} 

\begin{lemma}
    \label{lem:r-first-tours} 
Assume Condition \ref{cond:drift-for-coupling}, let $r$ be defined in 
\eqref{eq:def-r-n} and consider the coupling construction in
Definition \ref{def:coupling-kernel}.
Then, 
\[
    {\E}_{x,x',0}\bigg[\sum_{n=0}^{T_1} r(n)\bigg]
    \le 1 + \frac{r(1)}{\epsilon_b\phi(1)}\big(
       \bar{V}(x,x')-1
    \big)
    \charfun{(x,x)\notin\bar{C}}
    \label{eq:r-n-first}
\]
\end{lemma} 
\begin{proof} 
The claim holds trivially for $(x,x')\in\bar{C}$, so assume
$(x,x')\in\bar{C}$.
Lemma \ref{lem:joint-drifts} \eqref{eq:simple-drift} allows us to apply
Proposition \ref{prop:drift-sequence} with $\varphi =
\epsilon_b \phi$ and $b=\bar{b}$; note that $r_\varphi(n) = r(n)$. Then,
Proposition \ref{prop:comparison} with
$Z_k \defeq (H_k\circ\bar{V})(X_k,X_k')$ 
yields 
\[
    \E_{x,x',0}\bigg[\sum_{n=0}^{T_1-1} \epsilon_b \phi(1) r(n) 
    \bigg]
    \le H_0 \circ \bar{V}(x,x')
    =
    \bar{V}(x,x')-1.
\]
Lemma \ref{lem:log-concavity} 
\eqref{eq:log-concavity} implies
$r(n+1)\le r(1)r(n)$, so we deduce
\[
    {\E}_{x,x',0}\bigg[\sum_{n=0}^{T_1} r(n)\bigg]
    = 1 + {\E}_{x,x',0}\bigg[\sum_{n=1}^{T_1} r(n)\bigg]
    \le 1 + r(1) {\E}_{x,x',0}\bigg[
       \sum_{n=0}^{T_1-1} r(n)
     \bigg].
    \qedhere
\]
\end{proof} 

\begin{lemma}
    \label{lem:r-phi-sum} 
Assume Condition \ref{cond:drift-for-coupling}, 
let $r$ be defined in 
\eqref{eq:def-r-n} and consider the coupling construction in
Definition \ref{def:coupling-kernel}.
Then, 
\begin{align*}
    {\E}_{x,x',0}\bigg[\sum_{n=0}^{\tau-1} r(n)\bigg]
    \le 
    \frac{1}{\epsilon_b \phi(1)}\bigg[\bigg(1+\frac{c_* \bar{b}
      r^2(1)}{\epsilon_b \phi(1)} \bigg) \bar{V}(x,x')
    + \bigg( c_* \bar{b} r(1) -1 
    - \frac{c_* \bar{b} r^2(1)}{\epsilon_b \phi(1)}\bigg).
\end{align*}
where $\bar{b}$ and $c_*$ are given in Theorem \ref{thm:generic-explicit-rates}.
\end{lemma} 
\begin{proof} 
Lemma \ref{lem:joint-drifts} \eqref{eq:simple-drift} and
Propositions \ref{prop:comparison} and
\ref{prop:drift-sequence} applied 
as in the proof of Lemma \ref{lem:r-first-tours} yield
\[
    {\E}_{x,x',0}\bigg[\sum_{n=0}^{\tau - 1}
    \epsilon_b \phi(1)r(n)\bigg]
    \le \bar{V}(x,x')-1
    + \bar{b}{\E}_{x,x',0}\bigg[\sum_{n=0}^{\tau-1} r(n+1)
    \charfun{(X_n,X_n')\in\bar{C}}\bigg].
\]
The latter expectation can be written as
\begin{align*}
    {\E}_{x,x',0}\bigg[\sum_{n=0}^{\tau-1} 
    r(n+1)
    \charfun{(X_n,X_n')\in\bar{C}}\bigg]
    &= \sum_{j=1}^\infty{\E}_{x,x',0}\big[
    r(T_j+1) \chi_j
    \big],
\end{align*}
where $\chi_{j} \defeq \charfun{\tau > T_{j}} 
= \chi_{j-1} \charfun{D_{T_{j}+1}=0}$
for $j\ge 1$ and $\chi_0\equiv 1$.

Lemma \ref{lem:log-concavity} \eqref{eq:diff-est} 
implies for $j\ge 1$,
\begin{align*}
    r(T_{j+1}+1) \le
    r(T_j+1)\bigg(1+\epsilon_b(\phi'\circ
    H_\phi^{-1})\big(\epsilon_b(T_j+1)\big)\sum_{k=1}^{T_{j+1}-T_j} r(k)\bigg),
\end{align*}
and $\epsilon_b(\phi'\circ H_\phi^{-1})\big(\epsilon_b(T_j+1)\big)
\le \epsilon_b(\phi'\circ H_\phi^{-1})(\epsilon_b j) =
\delta_j$, because $T_j+1\ge j$ and $\phi'\circ H_\phi^{-1}$ is
non-increasing. 
We next show that, taking conditional expectation with respect to
$\F_{T_j} = \sigma\big((X_n,X_n',D_n): 1\le n\le T_j\big)$, we get
for $j\ge 1$
\begin{align}
    \E_{x,x',0}\big[&r(T_{j+1}+1) \chi_{j+1}\big] 
    \le {\E}_{x,x',0}\big[r(T_{j}+1)\chi_{j}\big]
    (1+
    \delta_j M_1)(1-\epsilon_\nu),
    \label{eq:cond-exp-stuff}
\end{align}
where $M_1$ is given below.
Namely, $\P_{x,x',0}(D_{T_j+1}=0\mid\F_{T_j})=(1-\epsilon_\nu)$ and
\begin{align*}
   \E_{x,x',0}\bigg[
    \sum_{k=1}^{T_{j+1}-T_{j}}
    r(k)\;\bigg|\;&\F_{T_j},D_{T_j+1}=0\bigg] \\
    &\le 
    \sup_{k\ge 1}
    \sup_{(x,x')\in\bar{C}}
    \int
    \check{P}_k(x,x';\ud y,\ud y') 
    \E_{y,y',0}^{(T_j)}\bigg[
      \sum_{n=0}^{T_1^{(T_j)}}
      r(n+1)\bigg],
\end{align*}
where $\E_{x,x',0}^{(j)}$ stands for the expectation
over the stopping time $T_1^{(j)}$ 
corresponding to the Markov chain $(X_n^{(j)},X_n'^{(j)},D_n^{(j)})$
constructed as in Definition \ref{def:coupling-kernel}
but defined using $(P_{k+j})_{k\ge 1}$ 
instead of $(P_k)_{k\ge 1}$. Lemma \ref{lem:r-first-tours} still
applies for this expectation, because it assumes only that
the kernels satisfy Condition 
\ref{cond:drift-for-coupling}.
Therefore Lemma \ref{lem:joint-drifts} \eqref{eq:joint-drift2} and
the bound $r(k+1)\le r(k)r(1)$ by Lemma
\ref{lem:log-concavity} \eqref{eq:log-concavity} yield
\begin{align*}
       \E_{x,x',0}\bigg[
    \sum_{k=1}^{T_{j+1}-T_{j}}
    r(k)\;\bigg|\;&\F_{T_j},D_{T_j+1}=0\bigg] 
    &\le r(1)\bigg[1 + \frac{2r(1)}{\epsilon_b \phi(1)} 
    \bigg(\frac{b_V+c_V}{1-\epsilon_\nu} -1\bigg)\bigg]
    =M_1.
\end{align*}

Applying \eqref{eq:cond-exp-stuff} recursively for $j+1,\ldots,2$ yields
\[
    \E_{x,x',0}\big[r(T_{j+1}+1) \chi_{j+1}\big] 
    \le (1-\epsilon_\nu)^{j-1}\prod_{i=1}^{j-1} (1+\delta_i M_1)
    {\E}_{x,x',0}\big[
    r_{\phi}(T_1+1) \big],
\]
and Lemma \ref{lem:r-first-tours} together with 
Lemma \ref{lem:log-concavity} \eqref{eq:log-concavity}
give
\[
    {\E}_{x,x',0}\big[
    r(T_1+1) \big]
    \le r(1)\bigg(1 + \frac{r(1)}{\epsilon_b\phi(1)}\big(
       \bar{V}(x,x')-1\big)\bigg).
\]
Putting everything together,
\[
{\E}_{x,x',0}\bigg[\sum_{n=0}^{\tau-1} r(n)\bigg]
    \le 
    \frac{1}{\epsilon_b\phi(1)}\bigg[\bar{V}(x,x')-1
    +
    c_*
    \bar{b}r(1)\bigg(1+\frac{r(1)}{\epsilon_b\phi(1)}
    \big(\bar{V}(x,x')-1\big)\bigg)\bigg],
\]
which equals the desired bound.
\end{proof} 


\section*{Acknowledgements} 

The work of C.\ Andrieu was partially supported by a Winton Capital research award.
M.\ Vihola was supported by the Academy of Finland
(project 250575).




\appendix

\section{Some results in the literature} 
\label{sec:literature} 

We restate here some results in the literature for the reader's
convenience. We start by stating
\cite[Proposition 11.3.2]{meyn-tweedie} for inhomogeneous
Markov chains; the proof of \cite{meyn-tweedie} applies without
modifications.
\begin{proposition}
    \label{prop:comparison} 
Suppose $(X_n)_{n\ge 0}$ is a Markov chain and let $\F_n \defeq
\sigma(X_0,\ldots,X_n)$ for $n\ge 0$. Assume $Z_n$ is non-negative and
$\F_n$-adapted, $f_n$ and $s_n$ are non-negative measurable functions and
\[
    \E[Z_{n+1} \mid \F_n] \le Z_n - f_n(X_n) + s_n(X_n)
    \qquad\text{for all $n\ge 0$.}
\]
Then, for any initial condition $x$ and any stopping time $\tau$
\[
    \E_x\bigg[\sum_{k=0}^{\tau-1} f_k(X_k) \bigg] \le Z_0(x) +
    \E_x\bigg[\sum_{k=0}^{\tau-1} s_k(X_k) \bigg].
\]
\end{proposition} 

\begin{proposition}[\protect{\cite[Proposition
        2.1]{douc-fort-moulines-soulier}}] 
    \label{prop:drift-sequence} 
Assume $P$ is a Markov kernel satisfying
\[
    P V(x) \le V(x) - \varphi\circ V(x) + b \charfun{x\in C},
\]
where $\varphi:[1,\infty)\to(0,\infty)$ is a non-decreasing convex function.
Then for 
$r_\varphi(k) \defeq (\varphi \circ H_\varphi^{-1})(n)/\varphi(1)$ where
$H_\varphi$ 
is as defined in \eqref{eq:def-r-n},
\[
    P V_{k+1}(x) \le V_k(x) - \varphi(1)r_\varphi(k) 
    + b r_\varphi(k+1) \charfun{x\in C}\qquad\text{for all $k\ge 0$},
\]
where $V_k\defeq H_k \circ V$ and 
\[
    H_k(v) \defeq \varphi(1)\int_0^{H_\varphi(v)}
r_\varphi(z+k)\ud z
= H_\varphi^{-1}(H_\varphi(v) + k)
-H_\varphi^{-1}(k).
\]
\end{proposition} 



\begin{thebibliography}{10} 

\bibitem{andrieu-fort-explicit}
C.~Andrieu and G.~Fort.
\newblock Explicit control of subgeometric ergodicity.
\newblock Research report 05:17, University of Bristol, 2005.

\bibitem{andrieu-moulines-priouret}
C.~Andrieu, {\'E}.~Moulines, and P.~Priouret.
\newblock Stability of stochastic approximation under verifiable conditions.
\newblock {\em {SIAM} J. Control Optim.}, 44(1):283--312, 2005.

\bibitem{andrieu-roberts}
C.~Andrieu and G.~O. Roberts.
\newblock The pseudo-marginal approach for efficient {M}onte {C}arlo
  computations.
\newblock {\em Ann. Statist.}, 37(2):697--725, 2009.

\bibitem{andrieu-thoms}
C.~Andrieu and J.~Thoms.
\newblock A tutorial on adaptive {MCMC}.
\newblock {\em Statist. Comput.}, 18(4):343--373, 2008.

\bibitem{andrieu-vihola}
C.~Andrieu and M.~Vihola.
\newblock Markovian stochastic approximation with expanding projections.
\newblock {\em Bernoulli}, 20(2):545--585, 2014.

\bibitem{andrieu-vihola-pseudo}
C.~Andrieu and M.~Vihola.
\newblock Convergence properties of pseudo-marginal {M}arkov chain {M}onte
  {C}arlo algorithms.
\newblock {\em Ann. Appl. Probab.}, to appear.
\newblock Preprint: arXiv:1210.1484v2.

\bibitem{atchade-fort}
Y.~Atchad{\'e} and G.~Fort.
\newblock Limit theorems for some adaptive {MCMC} algorithms with subgeometric
  kernels.
\newblock {\em Bernoulli}, 16(1):116--154, 2010.

\bibitem{baxendale}
P.~H. Baxendale.
\newblock Renewal theory and computable convergence rates for geometrically
  ergodic {M}arkov chains.
\newblock {\em Ann. Appl. Probab.}, 15(1A):700--738, 2005.

\bibitem{douc-fort-moulines-soulier}
R.~Douc, G.~Fort, E.~Moulines, and P.~Soulier.
\newblock Practical drift conditions for subgeometric rates of convergence.
\newblock {\em Ann. Appl. Probab.}, 14(3):1353--1377, 2004.

\bibitem{douc-moulines-rosenthal}
R.~Douc, E.~Moulines, and J.~S. Rosenthal.
\newblock Quantitative bounds on convergence of time-inhomogeneous {M}arkov
  chains.
\newblock {\em Ann. Appl. Probab.}, 14(4):1643--1665, 2004.

\bibitem{douc-moulines-soulier}
R.~Douc, E.~Moulines, and P.~Soulier.
\newblock Computable convergence rates for sub-geometric ergodic {M}arkov
  chains.
\newblock {\em Bernoulli}, 13(3):831--848, 2007.

\bibitem{fort-phd}
G.~Fort.
\newblock {\em Contrôle explicite d'ergodicité de chaînes de {M}arkov :
  application à l'analyse de convergence de l'algorithme {M}onte {C}arlo {EM}}.
\newblock PhD thesis, Univ. Paris VI, 2001.

\bibitem{fort-moulines-polynomial}
G.~Fort and E.~Moulines.
\newblock Polynomial ergodicity of {M}arkov transition kernels.
\newblock {\em Stochastic Process. Appl.}, 103(1):57--99, 2003.

\bibitem{fort-moulines-priouret}
G.~Fort, E.~Moulines, and P.~Priouret.
\newblock Convergence of adaptive and interacting {M}arkov chain {M}onte
  {C}arlo algorithms.
\newblock {\em Ann. Statist.}, 39(6):3262--3289, 2011.

\bibitem{jarner-roberts}
S.~F. Jarner and G.~O. Roberts.
\newblock Polynomial convergence rates of {M}arkov chains.
\newblock {\em Ann. Appl. Probab.}, 12(1):224--247, 2002.

\bibitem{maire-douc-olsson}
F.~Maire, R.~Douc, and J.~Olsson.
\newblock Partial ordering of inhomogeneous {M}arkov {C}hains with applications
  to {M}arkov chain {M}onte {C}arlo methods.
\newblock Preprint arXiv:1307.3719v2, 2013.

\bibitem{meyn-tweedie}
S.~Meyn and R.~L. Tweedie.
\newblock {\em Markov Chains and Stochastic Stability}.
\newblock Cambridge University Press, second edition, 2009.

\bibitem{meyn-tweedie-computable}
S.~P. Meyn and R.~L. Tweedie.
\newblock Computable bounds for geometric convergence rates of {M}arkov chains.
\newblock {\em Ann. Appl. Probab.}, 4(4):981--1011, 1994.

\bibitem{saksman-vihola}
E.~Saksman and M.~Vihola.
\newblock On the ergodicity of the adaptive {M}etropolis algorithm on unbounded
  domains.
\newblock {\em Ann. Appl. Probab.}, 20(6):2178--2203, 2010.

\bibitem{tuominen-tweedie}
P.~Tuominen and R.~L. Tweedie.
\newblock Subgeometric rates of convergence of {$f$}-ergodic {M}arkov chains.
\newblock {\em Adv. Appl. Probab.}, 26(3):775--798, 1994.

\end{thebibliography}
\end{document}